\documentclass[12pt]{amsart}

\usepackage{amssymb,amsmath,amsfonts,amsthm,mathrsfs}
\usepackage{fullpage}
\usepackage{thumbpdf}
\usepackage{ifthen}
\usepackage{ifpdf}
\ifpdf
    \usepackage[ pdftex, plainpages = false, pdfpagelabels, 
                 pdfpagelayout = useoutlines,
                 bookmarks,
                 bookmarksopen = true,
                 bookmarksnumbered = true,
                 breaklinks = true,
                 linktocpage,
                 pagebackref,
                 colorlinks = false,  % was true
                 linkcolor = blue,
                 urlcolor  = blue,
                 citecolor = red,
                 anchorcolor = green,
                 hyperindex = true,
                 hyperfigures,
                 pdftitle =  Galois \ Module \ Structure
                 ]{hyperref} 
\else
    \usepackage[ dvips, 
                 bookmarks,
                 bookmarksopen = true,
                 bookmarksnumbered = true,
                 breaklinks = true,
                 linktocpage,
                 pagebackref,
                 colorlinks = true,
                 linkcolor = blue,
                 urlcolor  = blue,
                 citecolor = red,
                 anchorcolor = green,
                 hyperindex = true,
                 hyperfigures
                 ]{hyperref}
\fi
\usepackage[all]{xy}
\usepackage{color}
\usepackage[lite, backrefs, alphabetic]{amsrefs}

\DeclareMathOperator{\Hom}{Hom}
\DeclareMathOperator{\Gal}{Gal}
\DeclareMathOperator{\Imm}{Im}

\DeclareMathOperator{\Char}{char}

\newtheorem{thm}{Theorem}

\newtheorem{prop}{Proposition}[section]
\newtheorem{lem}[prop]{Lemma}

\theoremstyle{definition}
\newtheorem{defn}[prop]{Definition}

\newtheorem{remark}[prop]{Remark}

\numberwithin{equation}{section}

\newcommand{\Z}{\mathbb{Z}}

\newcommand{\Ic}{\mathcal{I}}
\newcommand{\Gc}{\mathcal{G}}

\newcommand{\Gf}{\mathfrak{G}}

\newcommand{\Hf}{\mathfrak{H}}

\newcommand{\Gcm}[1]{\Gc^{ [#1] }}

\begin{document}

\begin{abstract}
In this paper we use the Merkurjev-Suslin theorem to determine the structure of arithmetically significant Galois modules that arise from Kummer theory.
Let $K$ be a field of characteristic different from a prime $\ell$, $n$ a positive integer, and suppose that $K$ contains the $(\ell^n)^{\text{th}}$ roots of unity.  
Let $L$ be the maximal $\ell^n$-elementary abelian extension of $K$, and set $G=\Gal(L| K)$.  
We consider the $G$-module $J:=L^\times/\ell^n$ and denote its socle series by $J_m$.
We provide a precise condition, in terms of a map to $H^3(G,\Z/\ell^n)$, determining which submodules of $J_{m-1}$ embed in cyclic modules generated by elements of $J_m$; therefore this map provides an explicit description of $J_m$ and $J_m/J_{m-1}$.
The description of $J_m/J_{m-1}$ is a new non-trivial variant of the classical Hilbert's Theorem 90.
The main theorem generalizes a theorem of Adem, Gao, Karaguezian, and Min\'a\v{c} which deals with the case $m=\ell^n=2$, and also ties in with current trends in minimalistic birational anabelian geometry over essentially arbitrary fields.

%This description of $J_m/J_{m-1}$ can be viewed as an analogue of the classical Hilbert's Theorem 90 and it is helpful for understanding the $G$-module $J$.
\end{abstract}

\title[Galois Module Structure]{Galois Module Structure of $(\ell^n)^{\text{th}}$ Classes of Fields}

\author[1]{J\'an Min\'a\v{c}$^{1}$}
\thanks{$^{1}$Research supported in part by NSERC grant R0370A01.}
\address{J\'{a}n Min\'{a}\v{c}, \vskip0pt
Department of Mathematics, \vskip0pt
Western University, \vskip0pt
1151 Richmond Street, \vskip0pt
London, Ontario \vskip0pt
Canada N6A 3K7}
\email{minac@uwo.ca}
\urladdr{http://www.math.uwo.ca/~minac/minac.html}

\author[2]{John Swallow$^{2}$}
\thanks{$^{2}$Research supported in part by NSF grant DMS-0600122.}
\address{John Swallow, \vskip0pt
Department of Mathematics and Computer Science, \vskip0pt
The University of the South, \vskip0pt
735 University Avenue, \vskip0pt
Sewanee, TN.  37383 \vskip0pt
USA}
\email{john.swallow@sewanee.edu}
%\urladdr{http://www.davidson.edu/math/swallow/}

\author[3]{Adam Topaz$^{3}$}
\thanks{$^{3}$Research supported in part by a Benjamin Franklin fellowship from the University of Pennsylvania and in part by NSF postdoctoral fellowship DMS-1304114.}
\address{Adam Topaz, \vskip0pt 
Department of Mathematics, \vskip0pt
University of California, Berkeley, \vskip0pt
970 Evans Hall \#3840, \vskip0pt
Berkeley, CA. \ 94720-3840 \ \vskip0pt
USA}
\email{atopaz@berkeley.edu}
%\urladdr{http://www.math.upenn.edu/~atopaz/}

\thanks{The authors would like to thank Ido Efrat for his kind encouragement and helpful comments related to the introduction.
The authors would also like to thank the referee for his careful reading of this paper and his useful suggestions regarding its exposition.}

\date{\today}
\maketitle

\section{Introduction}

\[\text{\bf``What information is encoded in Galois groups?''} \]
This is the fundamental question of Grothendieck's anabelian geometry; the several achievements in this subject show that in many special but important situations, the answer is ``everything'' when one deals with the all of the Galois theoretical information.
In the birational setting, for instance, the Neukirch-Uchida-Pop theorem  \cite{Neukirch1969}, \cite{Neukirch1969a}, \cite{Uchida1976}, \cite{Pop1994}, \cite{Pop2000} completely characterizes infinite finitely generated fields using their absolute Galois groups.

Current trends in the literature suggest that, in many cases, much of the arithmetic and geometry of the situation can already be detected using very minimal Galois theoretical information.
Most prominent is Bogomolov's program \cite{Bogomolov1991} in almost-abelian anabelian geometry -- for function fields over an algebraically closed field using the maximal pro-$\ell$ Galois group of a field -- whose aim is to reconstruct such function fields from their pro-$\ell$ abelian-by-central Galois groups.
While this program is far from being complete in its full generality, it has been carried through for function fields over the algebraic closure of a finite field \cite{Bogomolov2008a}, \cite{Bogomolov2011}, \cite{Pop2011}.

Similar results, however, are easily seen to be false for arbitrary fields -- indeed, there exist many non-isomorphic fields which have the same absolute Galois group.
It is still the case, however, that even very small pro-$\ell$ Galois groups still carry a significant amount of arithmetic/geometric information about the field, as described in the overview below.
In this paper, we find how much pro-$\ell$ Galois theoretical data is encoded in these minimal ``almost-abelian'' pro-$\ell$ Galois groups.
More precisely, we use the Merkurjev-Suslin Theorem to determine the structure of arithmetically significant Galois modules which form the building-blocks of meta-abelian pro-$\ell$ Galois groups. 
The main theorem of this paper is rather surprising as it shows that very small quotients of an absolute Galois group completely control the structure of Galois modules arising from {\bf strictly larger} Galois groups.
Moreover, it is rather surprising that this holds true for arbitrary fields which contain enough roots of unity.

\subsection{Overview}

Let $K$ be a field of characteristic different from a fixed prime $\ell$.
Assume that $K$ contains $\mu_{\ell^n}$, the $(\ell^n)^{\text{th}}$ roots of unity ($n \geq 1$ or $n = \infty$).
We denote by $K(\ell)$ the maximal pro-$\ell$ Galois extension of $K$ (inside a chosen separable closure) so that $\Gc_K := \Gal(K(\ell)|K)$ is the maximal pro-$\ell$ quotient of $G_K$, the absolute Galois group of $K$.
Let us recall the mod-$\ell^n$ central descending series of a pro-$\ell$ group $\Gc$:
\[ \Gc^{(1,n)} = \Gc, \ \ \Gc^{(m+1,n)} = [\Gc,\Gc^{(m,n)}] \cdot (\Gc^{(m,n)})^{\ell^n}. \]
To simplify the notation, we denote by $\Gcm{m,n} = \Gc/\Gc^{(m,n)}$; when no confusion is possible, we will omit the explicit $n$: $\Gc^{(m)} = \Gc^{(m,n)}$, $\Gcm{m}=\Gcm{m,n}$.
In the context of Galois theory, $\Gcm{3,n}_K$ is called the {\bf $\ell^n$-abelian-by-central} Galois group of $K$.

It has become increasingly evident that much of the arithmetic information of the field $K$ which is encoded in the pro-$\ell$ Galois group $\Gc_K$ can be recovered using the much smaller quotient $\Gcm{3}_K$.
For instance, in the case where $\ell^n = 2$, Min\'{a}\v{c}-Spira showed that the group $\Gcm{3}_K$ captures information about orderings of $K$, and that this group can be seen as a Galois-theoretical analogue of the Witt ring of quadratic forms of $K$ -- see \cite{Minac1990}, \cite{Minac1996}.

Furthermore, the Bloch-Kato conjecture -- now a theorem of Voevodsky-Rost et al. -- can be used to deduce that the cohomology ring of $G_K$ (and/or $\Gc_K$) $H^*(K,\Z/\ell^n(*))\cong H^*(K,\Z/\ell^n)$, along with the Bockstein morphism, can be recovered from the decomposable part of $H^*(\Gcm{3}_K,\Z/\ell^n)$; conversely, the group $\Gcm{3}_K$ can be recovered from $H^*(K,\Z/\ell^n)$ together with the Bockstein morphism -- see \cite{Chebolu2009} for details. 
Thus the group $\Gcm{3}_K$, for $\ell^n$ an arbitrary prime power, can be seen as a Galois-theoretical analogue of the Milnor K-ring $K_*^M(K)/\ell^n$.
Going further, in the recent paper \cite{Efrat2011c}, $\Gcm{3}_K$ was identified as the smallest quotient of $G_K$ for which the above holds. 
In any case, as an immediate consequence one can deduce the following:
Suppose that $S$ is a free pro-$\ell$ group;
then $\Gcm{3,1}_K \cong S^{[3,1]}$ if and only if $\Gc_K \cong S$.

On the other hand, $\Gc_K$ and its quotient $\Gcm{3,n}_K$ already encode similar valuation theoretic data of the field, for certain $n$.
The work of Jacob-Ware \cite{Jacob1989}, Engler-Nogueira \cite{Engler1994}, Efrat \cite{Efrat1995} and Engler-Koenigsmann \cite{Engler1998} in the 90's show that one can completely characterize the decomposition and inertia subgroups of $\Gc_K$ corresponding to tamely-branching valuations of $K$ using only the group-theoretical structure of $\Gc_K$.
Later on, it was realized by Efrat-Min\'a\v{c} \cite{Efrat2011a} that the much smaller quotient $\Gcm{3,1}_K$ suffices, in certain special cases, to detect tamely-branching {\bf $\ell$-Henselian} valuations of $K$ (see also \cite{Mah'e2004} for the $\ell=2$ case).
These results build upon the theory of rigid elements which was originally developed by Ware \cite{Ware1981}, then further developed by Arason-Elman-Jacob \cite{Arason1987}, Koenigsmann \cite{Koenigsmann1995}, Efrat \cite{Efrat1995}, \cite{Efrat1999}, and also others.
Similarly, Bogomolov and Tschinkel's theory of commuting-liftable pairs shows how to detect many more valuations using $\Gcm{3,\infty}_K$ but only for function fields $K|k$ over algebraically closed fields $k=\bar k$ -- see \cite{Bogomolov1991} and \cite{Bogomolov2007}.
Unifying the two methods mentioned above, the recent work of the third author \cite{Topaz2012} shows that the small quotients $\Gcm{3,n}_K$, for $n=1$ or $n=\infty$, already encode information about the decomposition/inertia subgroups of $\Gcm{2,n}_K$ corresponding to all valuations of $K$ where $K$ is an arbitrary field with $\mu_{\ell^n} \subset K$.
See e.g. \cite{Efrat2006}, \cite{Efrat2006b} and/or \cite{Topaz2012} for a comprehensive history and development of this subject within the work of the many authors mentioned above.

The theory of commuting-liftable pairs was used as the ``local-theory'' in the completion of Bogomolov's program in birational anabelian geometry over the algebraic closure of a finite field by Bogomolov-Tschinkel \cite{Bogomolov2008a}, \cite{Bogomolov2011} and separately by Pop \cite{Pop2010},\cite{Pop2011}.
In particular this shows that, if $K$ and $L$ are function fields over the algebraic closure of a finite field, then $G_K \cong G_L$ if and only if $\Gcm{3,\infty}_K \cong \Gcm{3,\infty}_L$, in a functorial way.

The examples above suggest the following question: to what extent is $\Gc_K$ determined by $\Gcm{3,n}_K$?
In this paper, we investigate a surprising relationship between the structure of small ``almost-abelian'' quotients of absolute Galois groups -- including and building upon $\Gcm{3}_K$ -- and the structure of Galois modules defined by {\bf strictly larger} Galois groups.
More precisely, we explore how these groups control certain arithmetically significant Galois modules $J$ which arise from Kummer theory; this is done by producing cohomological obstructions for determining cyclic submodules of $J$ (see below for details) which can be seen as higher versions of Hilbert's Theorem 90.
This provides a new and rather efficient way of describing important modules in Galois theory by combining powerful methods from Galois cohomology and techniques from modular representation theory.
The main result of this paper thus answers the question of determining the structure of the fundamental Galois module $J$. 
The work of this paper will have further applications towards the construction of $\ell^n$-meta-abelian Galois groups from much smaller $\ell^n$-abelian-by-central Galois groups of {\bf arbitrary fields} which contain sufficiently many roots of unity.
%This provides some evidence that a group-theoretical variant of Bogomolov's program \cite{Bogomolov1991} could be formulated in almost arbitrary situations.

\subsection{Notation}

All homomorphisms in the context of discrete and profinite groups will be continuous.
From now on, we will fix once and for all a prime $\ell$ and an integer $n \geq 1$.
For any abelian group $M$ we will denote by $M^\vee = \Hom(M,\Z/\ell^n)$.
Let $\Gc$ be a pro-$\ell$ group.
We denote by $H^*(\Gc) = H^*(\Gc,\Z/\ell^n)$, the continuous cochain cohomology of $\Gc$ with values in $\Z/\ell^n$.
Throughout we will denote by $H = \Gc^{(2)}$ and $G = \Gcm{2} = \Gc/H$ and assume that $G$ is $\Z/\ell^n$-torsion-free -- i.e. $G \cong \prod_i \Z/\ell^n$.
We consider the cohomology group $J := H^1(H)$ as a $G$-module.
Denoting $\Lambda := \Z/\ell^n[[G]]$ the completed group ring, we observe that $J$ is a continuous $\Lambda$-module where $J$ is given the discrete topology.

\begin{remark}
Let $K$ be a field as above ($\Char K \neq \ell$ and $\mu_{\ell^n} \subset K$).
Kummer theory implies that $\Gcm{2}_K$ is $\Z/\ell^n$-torsion-free -- i.e. $\Gcm{2}_K \cong \prod_i \Z/\ell^n$.

In this situation we pick, once and for all, an isomorphism of $G_K$-modules $\mu_{\ell^n} \cong \Z/\ell^n$ and use it tacitly throughout.
The module $J$ then has a significant arithmetical counterpart.
Indeed, recall that $H = \Gc_K^{(2)}$. 
Denoting by $L = K(\ell)^H$, we deduce that $L|K$ is Galois, $G = \Gal(L|K)$ and Kummer theory implies that $K(\sqrt[\ell^n]{K}) = L$.
Again by Kummer theory we deduce that $J = L^\times/\ell^n$, considered as a $G = \Gal(L|K)$-module where the action of $G$ on $J$ is compatible with the canonical projection $L^\times \twoheadrightarrow L^\times/\ell^n$.
\end{remark}

Going back to the general situation, denote by $\Ic$ the augmentation ideal of $\Lambda$.
The $\Lambda$-module $J$ has a canonical filtration (commonly known as the socle series) induced by $\Ic$ defined as follows:
\[J_m := \{ \gamma \in J \ : \ \eta \gamma = 0 \text{ for all } \eta \in \Ic^m\}. \]
Hence we immediately see that $J_m$ is a $\Lambda/\Ic^m$-module.
Thus, in particular $J_1 = J^G$ is the submodule of invariants; the higher $J_m$ can be seen as generalizations of the submodule of invariants in the sense that $J_{m+1}/J_m = (J/J_m)^G$; in particular, we deduce that $\Ic \cdot J_{m+1} \subset J_m$.
Moreover, this is an exhaustive filtration in the sense that $\bigcup_m J_m = J$.

Suppose $\gamma$ is some element of $J_{m}$.
Then there is a canonical $\Lambda$-homomorphism $\phi_\gamma \colon \Ic \rightarrow J_{m-1}$ defined by $\eta \mapsto \eta\gamma$.
This is, of course, the restriction to $\Ic$ of the unique $\Lambda$-linear map $\Lambda \rightarrow J_{m}$ defined by $1 \mapsto \gamma$.
In particular, we obtain a homomorphism of $\Z/\ell^n$-modules $J_{m} \rightarrow \Hom_G(\Ic,J_{m-1})$ defined by $\gamma \mapsto \phi_\gamma$.

A natural question arises: Which homomorphisms $\phi \colon \Ic \rightarrow J_{m-1}$ are actually defined by a $\gamma \in J_{m}$ as above; namely, what is the image of the canonical map $J_{m} \rightarrow \Hom(\Ic,J_{m-1})$.
Equivalently, one can ask: which submodules $M \leq J_{m-1}$ satisfy $\Ic \cdot \widetilde M = M$ for some cyclic submodule $\Lambda \cdot \gamma = \widetilde M \leq J_{m}$?
In some sense, this question provides a way to relate the $G$-module structure of $J_m$ to that of $J_{m-1}$.

In the special case $\ell^n=2$, $m=2$ and $\Gc = \Gc_K$ for a field $K$ of characteristic different from $2$, Adem, Gao, Karaguezian, and Min\'a\v{c} (\cite{AGKM} Theorem 4.1) determined a cohomological obstruction, with values in $H^3(G,\Z/\ell^n)$, which determines precisely which $\phi \colon \Ic \rightarrow J_1 = J^G$ arise from an element $\gamma \in J_{2}$.
In this paper, we provide a generalization of this obstruction as described below.

%%%%%%%%%%%%%%%%%%%%%%%%

Going back again to our general situation, given a homomorphism $\phi \colon \Ic \rightarrow J_m \leq J = H^1(H)$, we obtain an induced pairing:
\[ \Ic \times H \rightarrow \Z/\ell^n. \]
Denote by $H_\phi$ the right kernel of this pairing, then $H_\phi$ is a normal subgroup of $\Gc$ as $\Imm(\phi)$ is a $G$-submodule of $J$; we denote by $\Gc_\phi$ the quotient $\Gc/H_\phi$.

\begin{remark}
In the context of Galois theory -- i.e. $\Gc = \Gc_K$ and $L = K(\ell)^H$ as above -- take $\phi \colon \Ic \rightarrow L^\times/\ell^n = J$ a homomorphism (with image, for example, in $J_m$).  
We deduce using Kummer theory that $L(\sqrt[\ell^n]{\Imm\phi})|K$ is Galois and:
\[ \Gc_\phi = \Gal(L(\sqrt[\ell^n]{\Imm\phi})|K) \]
where $\Gc_\phi$ is the quotient of $\Gc$ as defined above.
\end{remark}

In this paper we give a naturally defined map, $\Psi \colon \Hom_G(\Ic, J_m) \rightarrow H^3(G,\Z/\ell^n)$ (Definition \ref{defn:psi}) which yields our desired obstruction for pro-$\ell$ groups $\Gc$ which satisfy an assumption reminiscent of the Merkurjev-Suslin theorem.
Moreover, we show that there is a group-theoretical recipe to compute $\Psi(\phi)$ using the group-theoretical structure of $\Gc_\phi$.

\begin{thm}
\label{thm:main-intro}
In the notation above, consider the following statements:
\begin{enumerate}
 \item $\phi = \phi_\gamma$ for some $\gamma \in J_{m}$ (see the definition of $\phi_\gamma$ above).
 \item $\Psi(\phi) = 0 \in H^3(G,\Z/\ell^n)$ (see Definition \ref{defn:psi} for the definition of $\Psi$).
\end{enumerate}
Then (1) implies (2).
Assume furthermore that the inflation map $H^2(G) \rightarrow H^2(\Gc)$ is surjective, then (1) and (2) are equivalent.
Moreover, there is a group-theoretical recipe to compute $\Psi(\phi)$ using the group-theoretical structure of $\Gc_\phi$.
\end{thm}

\begin{remark}
Let $K$ be a field as above and assume that $\Gc_K = \Gc$.
Then by the Merkurjev-Suslin theorem \cite{Merkurjev1982}, one has $H^2(\Gc_K,\Z/\ell^n) \cong H^2(K,\mu_{\ell^n}^{\otimes 2}) = K_2^M(K)/\ell^n$.
On the other hand, Kummer theory implies that $H^1(\Gcm{2}_K,\Z/\ell^n) \cong H^1(K,\mu_{\ell^n}) = K^\times/\ell^n$.
These isomorphisms are compatible with the cup product and so we deduce that the inflation map $H^2(\Gcm{2}_K) \rightarrow H^2(\Gc_K)$ is surjective so that $\Gc_K$ satisfies the added assumption of Theorem \ref{thm:main-intro}.
\end{remark}

\section{Encoding $J_m$ using Cohomology}

In this section and throughout the paper, we use the notation introduced in \S1.2.
In particular, $\Gc$ is a pro-$\ell$ group, $H = \Gc^{(2)}$ and $G = \Gc/H$.
For any $G$-module $M$, we consider $M$ also as a $\Gc$-module with action via the canonical projection $\Gc \rightarrow G$.
In particular, such a module $M$ is acted upon trivially by $H$.
If $N$ is another $G$-module, we endow $\Hom(M,N) = \Hom_\Z(M,N)$ with the usual structure of a $G$-module by defining: 
\[ (\sigma_* f)(m) = \sigma \cdot f(\sigma^{-1} \cdot m) \]
for $\sigma \in G$, $f \in \Hom(M,N)$ and $m \in M$.
Thus, $\Hom(M,N)^G = \Hom_G(M,N)$ is the set of $G$-equivariant homomorphisms from $M$ to $N$.

We now provide a cohomological way of encoding $J_m$ as $\Z/\ell^n$-modules.
We will also show how to encode the canonical map $J_{m} \rightarrow \Hom_G(\Ic,J_{m-1})$ using cohomology.
For simplicity, denote $\Ic_m := \Ic /\Ic^m$ and $\Lambda_m := \Lambda/\Ic^m$.
Recall that
\[ J_m = \{ \gamma \in J \ : \ \eta \cdot \gamma = 0 \ \forall \eta \in \Ic^m \}. \]
The canonical map $\Hom_G(\Lambda_m,J) \rightarrow J_m$ defined by $f \mapsto f(1)$ is an isomorphism (of $\Z/\ell^n$-modules).
Also, one has the obvious equality: $\Hom_G(\Ic_m,J) =  \Hom_G(\Ic,J_{m-1})$.
For $\gamma \in J_m$, we obtain a homomorphism $\phi_\gamma \colon \Ic \rightarrow J_{m-1}$ defined by $\eta \mapsto \eta \cdot \gamma$.
On the other hand, the map $\Ic_m \hookrightarrow \Lambda_m$ induces a canonical restriction map $\Hom_G(\Lambda_m,J) \rightarrow \Hom_G(\Ic_m,J)$; this map is precisely $\gamma \mapsto \phi_\gamma$ as defined above when one identifies $\Hom_G(\Lambda_m,J) \cong J_m$ as above.

Recall that, since $\Lambda_m^\vee$ is a trivial $H$-module, we have $H^1(H,\Lambda_m^\vee) = \Hom(H,\Lambda_m^\vee)$ and thus $H^1(H,\Lambda_m^\vee) = \Hom(\Lambda_m,J)$ by Pontryagin duality.
Therefore, $H^1(H,\Lambda_m^\vee)^G = \Hom_G(\Lambda_m,J)$ as described above; furthermore recall that $J_m \cong \Hom_G(\Lambda_m,J)$.
In a similar way, we deduce that $H^1(H,\Ic_m^\vee)^G = \Hom_G(\Ic_m,J)$.
The map $J_m \rightarrow \Hom_G(\Ic_m,J)$ defined by $\gamma \mapsto \phi_\gamma$ corresponds via these canonical isomorphisms precisely to the map $H^1(H,\Lambda_m^\vee)^G \rightarrow H^1(H,\Ic_m^\vee)^G$ induced by $\Lambda_m^\vee \rightarrow \Ic_m^\vee$, the dual of the canonical embedding $\Ic_m \hookrightarrow \Lambda_m$.
%By Pontryagin duality, we furthermore deduce that $H^1(H,\Lambda_m^\vee)^G = \Hom_G(\Lambda_m,J)$ while $H^1(H,\Ic_m^\vee)^G = \Hom_G(\Ic_m,J)$ and $\gamma \mapsto \phi_\gamma$ is precisely the map induced by $\Lambda_m^\vee \rightarrow \Ic_m^\vee$, the dual of the canonical embedding $\Ic_m \hookrightarrow \Lambda_m$.
We summarize this discussion for later use in the following lemma:

\begin{lem}
\label{lem:cohom-jm}
One has canonical isomorphisms:
\[ H^1(H,\Lambda_m^\vee)^G \xrightarrow{\cong} \Hom_G(\Lambda_m, J) \cong J_m, \]
and 
\[ H^1(H,\Ic_m^\vee)^G \xrightarrow{\cong} \Hom_G(\Ic_m,J) = \Hom_G(\Ic,J_{m-1}) \]
which are functorial in the natural sense.
Also, the following diagram commutes:
\[ 
\xymatrix{
H^1(H,\Lambda_m^\vee)^G \ar[d]\ar[r] & J_m \ar[d]^{\gamma \mapsto \phi_\gamma} \\
H^1(H,\Ic_m^\vee)^G \ar[r] & \Hom(\Ic,J_{m-1})
}
\]
\end{lem}

\begin{remark}
First, observe that the isomorphism $H^1(H,\Lambda_m^\vee)^G \xrightarrow{\cong} J_m$ is {\bf not} an isomorphism of $G$-modules.
Indeed, $G$ may act non-trivially on $J_m$ (the action is non-trivial as long as $J_1 \neq J_m$), but it acts trivially on $H^1(H,\Lambda_m^\vee)^G$.

In the Galois situation, if $\Gc = \Gc_K$ for a field $K$ as in the introduction, the lemma above can be seen as a kind of ``invariant Kummer theory'' in the following sense.
One has $H^1(H,\Z/\ell^n) = L^\times/\ell^n$ (using the notation of the introduction) -- this is an isomorphism of $G$-modules.
In particular, $J_1 = J^G = H^1(H,\Z/\ell^n)^G$.
On the other hand, $\Lambda_1 = \Z/\ell^n$ and thus $(\Lambda_1)^\vee = \Z/\ell^n$.
Generalizing to $\Ic^m \leq \Lambda$, we obtain the analogous isomorphism $H^1(H,\Lambda_m^\vee)^G \cong J_m$.
\end{remark}

%The goal of the main theorem of this paper is to determine the image of the canonical map $J_m \rightarrow \Hom(\Ic,J_{m-1})$ in a cohomological way.
In the main theorem, we determine the image of the canonical map $J_m \rightarrow \Hom(\Ic,J_{m-1})$ in a cohomological way.
Equivalently, using Lemma \ref{lem:cohom-jm}, we will compute the image of the canonical map:
\[ H^1(H,\Lambda_m^\vee)^G \rightarrow H^1(H,\Ic_m^\vee)^G. \]

\section{Cohomological Invariants}
\label{sec:cohom-invar}

In this section we will define the cohomological obstruction which ensures $\phi \colon \Ic \rightarrow J_{m-1}$ is defined by some $\gamma \in J$ as discussed above.

\begin{defn}
\label{defn:psi}
Consider the spectral sequence associated to the group extension $1 \rightarrow H \rightarrow \Gc \rightarrow G \rightarrow 1$:
\[ H^i(G,H^j(H,M)) \Rightarrow H^{i+j}(\Gc,M) \]
and the differential $d_2 \colon H^1(H,M)^G \rightarrow H^2(G,M)$.

On the other hand, consider the long-exact sequence in cohomology associated to the short exact sequence of $G$-modules:
\[ 0 \rightarrow \Z/\ell^n \rightarrow \Lambda_m^\vee \rightarrow \Ic_m^\vee \rightarrow 0. \]
 And in particular, consider the connecting homomorphism $\delta \colon H^2(G,\Ic_m^\vee) \rightarrow H^3(G,\Z/\ell^n)$.
We denote by $\Psi \colon H^1(H,\Ic_m^\vee) \rightarrow H^3(G,\Z/\ell^n)$ the composition of the two maps:
\[ \Psi \colon H^1(H,\Ic_m^\vee)^G \xrightarrow{d_2} H^2(G,\Ic_m^\vee) \xrightarrow{\delta} H^3(G,\Z/\ell^n). \]
\end{defn}

Before we proceed to prove the main theorem of the paper, we show that $\Psi(\phi)$ can be computed using the group-theoretical structure of $\Gc_\phi$.
First, let us recall a well-known characterization of a differential in the spectral sequence associated to a group extension.
Suppose $1 \rightarrow R \rightarrow \Hf \rightarrow \Gf \rightarrow 1$ is an arbitrary extension of pro-$\ell$ groups and let $M$ be an $\Gf$-module.
Consider the associated extension 
\[ \xi\colon \ 1 \rightarrow \frac{R}{[R,R]} \rightarrow \frac{\Hf}{[R,R]} \rightarrow \Gf \rightarrow 1 \] 
and note that $\Hf/[R,R]$ is the pushout in the following square:
\[ 
\xymatrix{
R \ar@{->>}[d] \ar[r] & \Hf \ar@{->>}[d] \\
R^{\rm{ab}} = \frac{R}{[R,R]} \ar[r] & \frac{\Hf}{[R,R]}
}
\]
The group extension $\xi$ defines a canonical element $\alpha_{\Hf} \in H^2(\Gf,R^{\rm{ab}})$.
On the other hand, one has a canonical pairing $R^{\rm{ab}} \times H^1(R,M) \rightarrow M$ since $H^1(R,M) = \Hom(R^{\rm{ab}},M)$ (recall that $R$ acts trivially on $M$).
Thus, we obtain a cup product:
\[ H^0(\Gf,H^1(R,M)) \times H^2(\Gf,R^{\rm{ab}}) \rightarrow H^2(\Gf,M).  \]
\begin{prop}
In the notation above, the differential 
\[d_2 \colon H^0(\Gf,H^1(R,M)) = E^{0,1}_2 \rightarrow E^{2,0}_2 = H^2(G,M)\]
corresponds to the pairing with $\alpha_{\Hf}$.
More precisely, $d_2(x) = -x \cup \alpha_{\Hf}$.
\end{prop}
\begin{proof}
See \cite{Neukirch2008} Theorem 2.4.4.
\end{proof}

%Let us now return to the main notation from this section.
Given an element $\phi \in \Hom_G(\Ic_m,J) = \Hom_G(\Ic,J_{m-1})$, we produce the corresponding group $\Gc_\phi$ as above (we denote by $\phi$ the corresponding element in $H^1(H,\Ic_m^\vee)^G$ following Lemma \ref{lem:cohom-jm}).
Recall that $\Gc_\phi$ fits as a group extension:
\[ 1 \rightarrow (\Imm \phi)^\vee \rightarrow \Gc_\phi \rightarrow G \rightarrow 1 \]
and that $(\Imm \phi)^\vee$ is abelian.
In particular, we obtain an element $\alpha_{\phi} := \alpha_{\Gc_\phi} \in H^2(G,(\Imm\phi)^\vee)$.
We denote by $\beta_\phi$ the image of $\alpha_\phi$ under the canonical map $H^2(G,(\Imm\phi)^\vee) \rightarrow H^2(G,\Ic^\vee)$.
Thus, one has the following equality of the differential $d_2(\phi) = -\phi \cup \beta_\phi \in H^2(G,\Ic^\vee)$ by the proposition above, along with the functoriality of the spectral sequences under consideration.
Clearly, the boundary morphism $\delta \colon H^2(G,\Ic^\vee) \rightarrow H^2(G,\Z/\ell^n)$ can be computed using only $G$ so that the composition $\Psi(\phi) = \delta(-\phi \cup \beta_\phi)$ can be computed using $\Gc_\phi$ along with the projection $\Gc_\phi \twoheadrightarrow G$.
Moreover, $G = \Gcm{2}_\phi = \Gc_\phi/\Gc_\phi^{(2)}$ and so the quotient $\Gc_\phi \twoheadrightarrow G$ can be computed group-theoretically, as required.

\begin{remark}
\label{remark:Psi-vs-agkm}
We compare $\Psi$ with the formula given in \cite{AGKM} Theorem 4.1.
Recall that $\Ic_2 = \Ic/\Ic^2$.
In this case, we can provide an explicit alternative formula for $\Psi(\phi)$ as follows.
Let $(x_i)_i$ be a $\Z/\ell^n$-basis for $H^1(G)$ and consider the minimal convergent generating set $(\sigma_i)_i$ for $G$ which is dual to $(x_i)_i$.
Then $\Ic = \langle(\sigma_i - 1) \rangle_i$, and $\Ic/\Ic^2$ is a trivial $G$-module, isomorphic to $\prod_i \Z/\ell^n $ with a minimal topological $\Z/\ell^n$-generating set given by $(\sigma_i-1) =: \rho_i$; namely, the canonical map \[\eta \colon \prod_i (\Z/\ell^n) \cdot \rho_i \rightarrow \Ic/\Ic^2\] is an isomorphism.
Indeed, $\sigma \cdot (\tau-1) = (\tau-1) \mod \Ic^2$ since $(\tau-1)(\sigma-1) \in \Ic^2$ -- thus $\Ic/\Ic^2$ is a trivial module.
Observe that $(\sigma\tau-1) = (\sigma-1)(\tau-1)+(\sigma-1)+(\tau-1)$ and so indeed $(\rho_i)_i$ are generators of $\Ic/\Ic^2$.
Moreover, one easily deduces that these are free generators for $\Ic/\Ic^2$ since the $(\sigma_i)$ are a minimal generating set for $G$ and we explicitly assume that $G$ is $(\Z/\ell^n)$-torsion-free.

Thus, $H^1(H,\Ic_2^\vee)^G \cong \Hom_G(\Ic/\Ic^2,J) \rightarrow \bigoplus_i J^G$ defined by $\phi \mapsto (\phi(\rho_i))_i$ is an isomorphism.
On the other hand, one has $J^G = H^1(H)^G$.
By functoriality, the following diagram commutes:
\[
\xymatrix{
H^1(H,\Ic_2^\vee)^G \ar[r]^{d_2} \ar[d]_\cong & H^2(G,\Ic_2^\vee) \ar[d]^{\cong} \\
\bigoplus_i H^1(H,\Z/\ell^n)^G \ar[r]_{\bigoplus_i d_2} & \bigoplus_i H^2(G,\Z/\ell^n)
}\]
where the vertical isomorphisms are induced by the isomorphism $\eta \colon \prod_i (\Z/\ell^n) \cdot \rho_i \rightarrow \Ic/\Ic^2$.
What remains is to calculate the boundary homomorphism $\delta \colon H^2(G,\Ic_2^\vee) \rightarrow H^3(G,\Z/\ell^n)$ which we do in the following lemma:
\begin{lem}
\label{lem:psi-vs-agkm}
Let $(\sigma_i)_i$ be a minimal generating set for $G$ with dual basis $(x_i)_i$ for $H^1(G,\Z/\ell^n)$.
Then $\Ic_2 = \Ic/\Ic^2$ is isomorphic as a (trivial) $G$-module to the module:
\[ \prod_i (\Z/\ell^n) \cdot \rho_i \]
where $\rho_i = \sigma_i-1$.
Under this identification, one has an isomorphism:
\[ H^2(G,\Ic_2^\vee) \rightarrow \bigoplus_i H^2(G,\Z/\ell^n) \]
defined on the level of cocycles by
\[ \xi(\bullet,\bullet) \mapsto \sum_i \xi(\bullet,\bullet)(\rho_i). \]
Via this isomorphism, the connecting homomorphism $\delta \colon \bigoplus_i H^2(G,\Z/\ell^n) \cong H^2(G,\Ic_2^\vee)  \rightarrow H^3(G,\Z/\ell^n)$ is given by:
\[ \delta\left(\sum_i \xi_i\right) = \sum_i -x_i \cup \xi_i. \]
\end{lem}
\begin{proof}
The discussion preceding the statement of the lemma proves the first two claims.
Let us now calculate $\delta$.
Observe that the short exact sequence
\[ 1 \rightarrow \Z/\ell^n \rightarrow \Lambda_2^\vee \rightarrow \Ic_2^\vee \rightarrow 1 \]
is split by a section $\widetilde{(\bullet)} \colon \Ic_2^\vee \rightarrow \Lambda_2^\vee$ which is a $\Z/\ell^n$-linear homomorphism sending $f \colon \Ic_2 \rightarrow \Z/\ell^n$ to the unique homomorphism $\Lambda_2 \rightarrow \Z/\ell^n$ defined by 
\[ \tilde f(\rho_i) = f(\rho_i), \ \ \tilde f(1) = 0. \]
Now take a cocycle $\xi \colon G\times G \rightarrow \Ic_2^\vee$ and consider the map $\tilde\xi \colon G \times G \rightarrow \Lambda_2^\vee$ defined by $\tilde \xi(\sigma,\tau) = \widetilde {\xi(\sigma,\tau)}$.
Then $\delta(\xi) = d\tilde\xi$ has a representing cocycle given by:
\begin{eqnarray*}
\delta(\xi)(a,b,c) &=& [a\tilde\xi(b,c)](1)-[\tilde\xi(ab,c)](1)+[\tilde\xi(a,bc)](1)-[\tilde\xi(a,b)](1) \\
 &=& [\tilde\xi(b,c)](a^{-1})-[\tilde\xi(b,c)](1) \\
 &=& [\tilde\xi(b,c)](a^{-1}-1) \\
 &=& -[\tilde\xi(b,c)](a-1)
\end{eqnarray*}
Thus, writing $a = \prod_i \sigma_i^{a_i}$ we have using the discussion of Remark \ref{remark:Psi-vs-agkm}:
\begin{eqnarray*}
 \delta(\xi)(a,b,c) &=& -[\tilde\xi(b,c)](a-1) \\
 &=& -[\tilde\xi(b,c)]\left(\sum_i a_i \rho_i\right) \\ 
 &=& \sum_i -a_i\xi(b,c)(\rho_i)
\end{eqnarray*}
and thus we've deduced that:
\[ \delta(\xi) = \sum_i -x_i \cup [\xi(\bullet,\bullet)](\rho_i). \]
as required.
\end{proof}

Under the identifications in the discussion above:
\begin{itemize}
\item $\Hom(\Ic,J_1) = H^1(H,\Ic_2^\vee)^G$
\item $J_1 = H^1(H,\Z/\ell^n)^G$
\end{itemize}
We deduce that the map $\Psi \colon \Hom(\Ic,J_1) \rightarrow H^3(G,\Z/\ell^n)$ is given by:
\[ \Psi(\phi) = \sum_i -x_i \cup d_2(\phi(\rho_i)). \]

If $\ell^n = 2$, $\Gc = \Gc_K$ for an arbitrary field $K$ of characteristic different from $2$, this formula for $\Psi(\phi)$ is precisely the one given in \cite{AGKM} Theorem 4.1, if we consider $\phi \colon \Ic \rightarrow J_1$ as a sequence $(\gamma_i)_i$ indexed by $\rho_i$ where almost all $\gamma_i = 0$ as in loc.cit.
\end{remark}

\section{Proof of Main Theorem}

Using the discussion above, we can rephrase Theorem \ref{thm:main-intro} as follows.

\begin{thm}
\label{thm:main-thm}
Let $\Gc$ be a pro-$\ell$ group and $G = \Gcm{2} = \Gc/\Gc^{(2)}$ as above (i.e. $G \cong \prod_i \Z/\ell^n$).
Let $\phi \in H^1(H,\Ic_m^\vee)^G$ be given.
Consider the following statements:
\begin{enumerate}
\item There exists $\gamma \in H^1(H,\Lambda_m^\vee)^G$ such that $\phi = \phi_\gamma$ -- i.e. $\phi$ is contained in the image of the canonical map $H^1(H,\Lambda_m^\vee)^G \rightarrow H^1(H,\Ic_m^\vee)^G$.
 \item $\Psi(\phi) = 0 \in H^3(G,\Z/\ell^n)$ (see Definition \ref{defn:psi} for the definition of $\Psi$).
\end{enumerate}
Then (1) implies (2).
Assume furthermore that the inflation map $H^2(G) \rightarrow H^2(\Gc)$ is surjective, then (1) and (2) are equivalent.
\end{thm}

The fact that $\Psi(\phi)$ can be computed using $\Gc_\phi$ was shown in \S\ref{sec:cohom-invar}.
Furthermore, the fact that (1) implies (2) is trivial by the definition of $\Psi$ (see e.g. the bottom-left corner of the diagram of Lemma \ref{lem:comm-diagram}).
The remainder of the paper will be devoted to showing that (2) implies (1) under the added assumption that the inflation $H^2(G) \rightarrow H^2(\Gc)$ is surjective.

\begin{lem}
\label{lem:comm-diagram}
The following diagram is commutative with exact rows and columns:
\[
\xymatrix{
{} & H^2(G,\Z/\ell^n) \ar[d]\ar[r] & H^2(\Gc,\Z/\ell^n) \ar[d]\ar[r] & 0 \\
H^1(H,\Lambda_m^\vee)^G \ar[d]\ar[r]^{d_2} & H^2(G,\Lambda_m^\vee) \ar[d]\ar[r] & H^2(\Gc,\Lambda_m^\vee)\ar[d] \\
H^1(H,\Ic_m^\vee)^G \ar[dr]_\Psi \ar[r]^{d_2} & H^2(G,\Ic_m^\vee) \ar[d]\ar[r] & H^2(\Gc,\Ic_m^\vee) \\
{} & H^3(G,\Z/\ell^n)
}
\]
\end{lem}
\begin{proof}
The commutativity of the diagram follows from the functoriality of the exact sequences in cohomology.
The vertical columns are portions of the long exact sequences in cohomology while the rows are portions of the exact sequence arising from the spectral sequence associated to the group extension
\[ 1 \rightarrow H \rightarrow \Gc \rightarrow G \rightarrow 1. \]
The surjectivity of the top row is the explicit assumption in the main theorem.
\end{proof}
A diagram chase using the diagram in Lemma \ref{lem:comm-diagram} shows that for $\phi \in H^1(H,\Ic_m^\vee)^G$ the following are equivalent:
\begin{itemize}
 \item $\Psi(\phi) = 0$.
 \item There exists some $\phi' \in \ker(d_2 \colon H^1(H,\Ic_m^\vee)^G \rightarrow H^2(G,\Ic_m))$ such that $\phi' + \phi$ is in the image of $H^1(H,\Lambda_m^\vee)^G \rightarrow H^1(H,\Ic_m^\vee)^G$.
\end{itemize}

We have thus reduced the proof to the following proposition.

\begin{prop}
 \label{prop:ker-d2}
The kernel of $d_2 \colon H^1(H,\Ic_m^\vee)^G \rightarrow H^2(G,\Ic_m^\vee)$ is contained in the image of $H^1(H,\Lambda_m^\vee)^G \rightarrow H^1(H,\Ic_m^\vee)^G$.
\end{prop}
\begin{proof}
Let us first note that $(\Ic^{m-1}/\Ic^m)^\vee$ is a trivial $G$-module with $\ell^n$ torsion as an abelian group.
Therefore the inflation map $H^1(G,(\Ic^{m-1}/\Ic^m)^\vee) \rightarrow H^1(\Gc,(\Ic^{m-1}/\Ic^m)^\vee)$ is an isomorphism; indeed, both groups are isomorphic to $\Hom(G, (\Ic^{m-1}/\Ic^m)^\vee)$ since $G = \Gcm{2}$.
Thus
\[\ker(d_2 \colon H^1(H,(\Ic^{m-1}/\Ic^m)^\vee)^G \rightarrow H^2(G,(\Ic^{m-1}/\Ic^m)^\vee)) = 0.\]

Now we proceed to prove the Proposition by descending induction on $m$. The base case $m = 2$ is obtained from the observation above since, in particular, the map $d_2$ is injective.
Suppose inductively that the kernel of $d_2 \colon H^1(H,\Ic_m^\vee)^G \rightarrow H^2(G,\Ic_m^\vee)$ lies in the image of $H^1(H,\Lambda_m^\vee)^G$.
In particular, Theorem \ref{thm:main-thm} holds for $m$ by our inductive hypothesis.
We shall first show that any $\phi \in \ker(d_2 \colon H^1(H,\Ic_{m+1}^\vee)^G \rightarrow H^2(G,\Ic_{m+1}^\vee))$ actually lies in the image of the canonical map $H^1(H,\Ic_m^\vee)^G \rightarrow H^1(H,\Ic_{m+1}^\vee)^G$.
This follows from a diagram chase using the following diagram with exact rows and columns. Note, in particular, that the restriction map $H^1(\Gc,(\Ic^m/\Ic^{m+1})^\vee) \rightarrow H^1(H, (\Ic^m/\Ic^{m+1})^\vee)^G$ is the trivial map as noted above:
\[
\xymatrix{
 & H^1(H, \Ic_m^\vee)^G \ar[d] & \\
 & H^1(H, \Ic_{m+1}^\vee)^G \ar[r]^{d_2} \ar[d] & H^2(G,\Ic_{m+1}^\vee) \ar[d] \\
0 \ar[r] & H^1(H, (\Ic^m/\Ic^{m+1})^\vee)^G \ar[r]^{d_2} & H^2(G,(\Ic^m/\Ic^{m+1})^\vee)
}
\]
Note that exactness of the first column follows from the fact that the functor $M\to M^G$ is left exact and it is applied to the exact sequence:
\[0 \rightarrow H^1(H,\Ic_m^\vee)^G \rightarrow H^1(H,\Ic_{m+1}^\vee)^G \rightarrow H^1(H, (\Ic^m/\Ic^{m+1})^\vee)^G \]
We have shown that, if $\phi \in H^1(H, \Ic_{m+1}^\vee)^G$ lies in the kernel of $d_2$, then $\phi$ is the image of some element $\phi_0 \in H^1(H,\Ic_m^\vee)^G$.
By the functoriality of $\Psi$ we see that $\Psi(\phi_0)$ is zero.
By Theorem \ref{thm:main-thm} which inductively holds $m$, we deduce that $\phi_0$ is the image of some element $\psi \in H^1(H,\Lambda_m^\vee)^G$.
In particular, $\phi$ must be the image of $\psi$ under the canonical map
\[ H^1(H,\Lambda_m^\vee)^G \rightarrow H^1(H,\Lambda_{m+1}^\vee)^G \rightarrow H^1(H,\Ic_{m+1}^\vee)^G. \qedhere\]
\end{proof}

This completes the proof of Theorem \ref{thm:main-thm}.
Once we identify $H^1(H,\Ic_m^\vee)^G$ with $\Hom_G(\Ic_m,J)$ and $H^1(H,\Lambda_m^\vee)^G$ with $\Hom_G(\Lambda_m,J)$ as in Lemma \ref{lem:cohom-jm}, we deduce Theorem \ref{thm:main-intro}.


\end{document}